\documentclass[a4, 12pt, 
]{amsart}

\usepackage[all]{xy}

\usepackage[dvipdfmx]{hyperref}
\newtheorem{theo}{Theorem}[section]

\newtheorem{lemm}[theo]{Lemma}

\numberwithin{equation}{section}

\theoremstyle{definition}

\theoremstyle{remark}
\newtheorem{rem}[theo]{Remark}

\newcommand{\Ker}[0]{\operatorname{Ker}}

\newcommand{\Hom}[0]{\operatorname{Hom}}

\newcommand{\codim}[0]{\operatorname{codim}}
\newcommand{\Aut}[0]{\operatorname{Aut}}
\newcommand{\pr}[0]{\operatorname{pr}}

\newcommand{\deldel}{\sqrt{-1}\partial \overline{\partial}}

\newcommand{\simm}{\hspace{-0.1cm}\sim}

\newcommand{\Rur}[3]{R_{#1}(#2, \bar{#2}, #3, \bar{#3})}

\usepackage{geometry}
\usepackage{setspace} 
\begin{document}

\date{\today, version 0.01}

\title[On semi-positive holomorphic sectional curvature]
{On projective manifolds with \\semi-positive holomorphic sectional curvature}

\author{Shin-ichi MATSUMURA}
\email{{\tt mshinichi-math@tohoku.ac.jp, mshinichi0@gmail.com}}
\address{Mathematical Institute, Tohoku University, 
6-3, Aramaki Aza-Aoba, Aoba-ku, Sendai 980-8578, Japan.}

\renewcommand{\subjclassname}{%
\textup{2010} Mathematics Subject Classification}


\maketitle

\begin{abstract}
In this paper,     
we establish a structure theorem for a smooth projective variety $X$ 
with semi-positive holomorphic sectional curvature. 
Our structure theorem contains the solution for Yau's conjecture and it can be regarded as  
a natural generalization of the structure theorem proved by Howard-Smyth-Wu and Mok 
for holomorphic bisectional curvature.    
Specifically, we prove that $X$ admits a locally trivial morphism $\phi:X \to Y$ 
such that the fiber $F$ is rationally connected and 
the image $Y$ has a finite \'etale cover $A \to Y$ by an abelian variety $A$, 
by combining the author's previous work with the theory of holomorphic foliations. 
Moreover, we show that the universal cover of $X$ is biholomorphic and isometric  
to the product $\mathbb{C}^m \times F$ of the universal cover $\mathbb{C}^m$ of $Y$ with a flat metric 
and the rationally connected fiber $F$ 
with a K\"ahler metric whose holomorphic sectional curvature is semi-positive. 
\end{abstract}          
 
%\tableofcontents

\section{Introduction}\label{sec1}

After Siu-Yau and Mori had solved the Frankel conjecture in \cite{SY80} and \cite{Mor79}, 
Howard-Smyth-Wu and Mok  established the structure theorem 
for a compact K\"ahler manifold $M$ with semi-positive holomorphic bisectional curvature 
in \cite{HSW81} and \cite{Mok88} (see also \cite{CG71} and \cite{CG72}). 
Their structure theorem states that $M$ admits a locally trivial morphism $f: M\to B$ 
such that the fiber $F$ is a Hermitian symmetric manifold and 
the image $B$ has a finite \'etale cover $T \to B$ by a complex torus $T$. 
Moreover,  it says that 
the universal cover of $M$ is biholomorphic and isometric to the product $\mathbb{C}^{m} \times M'$.

This paper is devoted to studies of (semi-)positivity of holomorphic sectional curvature. 
Positivity of holomorphic sectional curvature is 
much weaker than that of holomorphic bisectional curvature. 
In fact, by Yang's affirmative answer for Yau's conjecture posed in \cite{Yau82}, 
we know that a compact K\"ahler manifold $X$ with positive holomorphic sectional curvature 
is rationally connected (see \cite{Yan18a} and see \cite{HW15}, \cite{Mat18b} for another approach). 
However, the manifold $X$ can not be expected to be a Hermitian symmetric manifold 
and even to have the nef anti-canonical bundle (see \cite[Example 3.6]{Yan16}). 
Nevertheless, holomorphic bisectional curvature (more generally, the curvature tensor) 
is actually determined by holomorphic sectional curvature. 
%The proof of this fact depends on non-trivial way  
%and there seems to be no direct relation of positivity of them. 
%The proof is based on an implicit way, 
%and thus a direct relation between them can not be obtained. 
Therefore it is an interesting and natural problem to 
investigate a relation or an analogy between holomorphic sectional curvature and  bisectional curvature.
Particularly, it is one of the most significant problems 
to generalize the structure theorem of Howard-Smyth-Wu and Mok to semi-positive holomorphic sectional curvature.

The main result of this paper is the following structure theorem 
for smooth projective varieties with semi-positive holomorphic sectional curvature. 
Our theorem affirmatively solves \cite[Conjecture 1.3]{Mat18b} in a strong form. 
Moreover, the theorem naturally contains the solution for Yau's conjecture 
and it provides a generalization of 
the structure theorem of Howard-Smyth-Wu and Mok in the case of projective varieties.

\begin{theo}\label{m-thm}
Let $X$ be a smooth projective variety 
equipped with a K\"ahler metric $g$ with semi-positive holomorphic sectional curvature. 
Then we have the following statements\,$:$
\vspace{0.2cm}\\
$(1)$ 
There exists a smooth morphism $\phi: X \to Y$ 
to a smooth projective variety $Y$ with the following properties\,$:$
\begin{itemize}
\item The morphism $\phi: X \to Y$ is a locally trivial morphism 
$($that is, all the fibers are isomorphic$)$. 
\item The image $Y$ is a smooth projective variety with a flat metric. 
In particular, there exists a finite \'etale cover $A \to Y$ by an abelian variety $A$. 
\item The fiber $F$ is a rationally connected manifold. 
In particular, the morphism $\phi: X \to Y$ is a MRC $($maximal rationally connected$)$ fibration of $X$. 
\end{itemize}
%Moreover the fiber product $\widetilde{X}:=X \times_Y A$ admits
%isomorphic to the product $A \times F$ of the abelian variety $A$ and 
%the rationally connected fiber $F$. 
Therefore the fiber product $X^*:=A \times_Y X$ admits 
the locally trivial Albanese map $X^* \to A$ to the abelian variety $A$ 
with the rationally connected fiber $F$. 
\vspace{0.2cm}\\
$(2)$ 
Moreover we obtain the isomorphism 
$$
X_{\rm{univ}} \cong \mathbb{C}^m \times F, 
$$
where $X_{\rm{univ}}$ is the universal cover of $X$ and 
$F$ is the rationally connected fiber of $\phi$. 
We have the following commutative diagram\,$:$
\begin{equation*}
\xymatrix@C=40pt@R=30pt{
 X_{\rm univ} \cong \mathbb{C}^{m} \times F\ar[r]^{} \ar[d]^{} & X^*:=A \times_Y X\ar[d] \ar[r]^{} & X\ar[d]^{\phi} \\ 
\mathbb{C}^{m} \ar[r]  & A \ar[r]^{}  &  Y.\\   
}
\vspace{0.1cm}
\end{equation*}
In particular, the fundamental group of $X$ is an extension of a finite group by $\mathbb{Z}^{2m}$. 
\vspace{0.2cm}\\
$(3)$ There exist a K\"ahler metric $g_F$ on the fiber $F$ and 
a K\"ahler metric $g_Y$ on $Y$ with the following properties\,$:$

\begin{itemize}
\item The holomorphic sectional curvature of $g_F$ is semi-positive.  
\item The K\"ahler metric $g_Y$ is flat. 
\item The above isomorphism $X_{\rm{univ}} \cong \mathbb{C}^m \times F$ is not only biholomorphic but also isometric 
with respect to the K\"ahler metrics $\mu^*g $, $\pi^*g_{Y}$, and $g_F$. 
\end{itemize}
Here $\pi$ and $\mu$ respectively denote the universal cover $\pi: \mathbb{C}^m \to Y$ of $Y$ 
and the universal cover $\mu: X_{\rm univ} \to X$ of $X$. 
\end{theo}

The proof of Theorem \ref{m-thm} is based on 
a combination of studies of MRC fibrations developed in \cite{Mat18b} 
and the theory of (holomorphic) foliations. 
It is worth to mention that a part of \cite{Mat18b} is strongly influenced by 
the idea in \cite{HW15} and the notion of RC positivity  in \cite{Yan18a}, 
which are motivated by Yau's conjecture in \cite{Yau82} 
(see \cite{Yan18b} and \cite{Yan18c} for RC positivity). 
Theorem \ref{m-thm} can be derived from Theorem \ref{main-thm} and Theorem \ref{mainn-thm}. 

The following theorem, which is one of the key ingredients of this paper, 
is proved by the splitting theorem of the (holomorphic) tangent bundle $T_X$ of $X$ 
(see Theorem \ref{key-thm}) 
and the theory of foliations with rationally connected leaves developed in \cite{Hor07}.

\begin{theo}\label{main-thm}
Let $X$ be a smooth projective variety admitting a K\"ahler metric  with semi-positive holomorphic sectional curvature.

Then we can choose a MRC fibration $X \to Y$ of $X$ to be 
a morphism $($without indeterminacy locus$)$  to a smooth projective variety $Y$. 
\end{theo}

The following theorem, which can be seen as a generalization of the main result in \cite{Mat18b}, 
reveals detailed properties of a morphism $\phi: X \to Y$ 
to a compact K\"ahler manifold $Y$ with pseudo-effective canonical bundle. 
The proof is based on the theory of foliations and Ehresmann's theorem. 
By applying Theorem \ref{mainn-thm} to the MRC fibration $\phi: X \to Y$ 
chosen to be a morphism by Theorem \ref{main-thm}, 
we can obtain all the conclusions in Theorem \ref{m-thm}. 
In our situation, the fiber $F$ is rationally connected  (in particular, simply connected), 
and thus it is sufficient for Theorem \ref{m-thm} to consider only the restricted case. 
However, we formulate Theorem \ref{mainn-thm} in a more general statement, 
since it seems to be important to treat Albanese maps (instead of MRC fibrations)  
in order to generalize Theorem  \ref{m-thm} to non-projective varieties.

\begin{theo}\label{mainn-thm}
Let $(X, g)$ be a compact K\"ahler manifold with semi-positive holomorphic sectional curvature, and 
let $\phi: X \to Y$ be a morphism to a compact K\"ahler manifold $Y$ with pseudo-effective canonical bundle. 
\vspace{0.2cm}\\
$(1)$ The following statements hold\,$:$
\begin{itemize}
\item The morphism $\phi: X \to Y$ is smooth and locally trivial. 
\item The exact sequence of vector bundles 
$$
0 \xrightarrow{\quad \quad} T_{X/Y} 
\xrightarrow{\quad \quad} T_X \xrightarrow{\quad d\phi_* \quad} \phi^{*} T_Y 
\xrightarrow{\quad \quad}  0
$$
admits the holomorphic orthogonal splitting, 
that is, the orthogonal complement $T_{X/Y}^{\bot}$ is a holomorphic vector bundle and 
there exists an isomorphism $j: \phi^{*} T_Y \to T_{X/Y}^{\bot}$ 
such that it gives the holomorphic orthogonal decomposition
$$
T_X = T_{X/Y} \oplus j(\phi^{*} T_Y) \cong T_{X/Y} \oplus \phi^{*} T_Y.
$$ 
\item The image $Y$ admits a K\"ahler metric $g_{Y}$ 
such that $g_Q= \phi^{*} g_{Y}$ and the holomorphic sectional curvature of $g_Y$ is identically zero. 
In particular, the image $Y$ has a finite \'etale cover $T \to Y$ by a complex torus $T$. 
Here $g_Q$ is the hermitian metric on $\phi^{*} T_Y$ 
induced by the above exact sequence and the metric $g$. 
\end{itemize}
$(2)$ We obtain the isomorphism 
$$
X_{\rm{univ}} \cong \mathbb{C}^m \times F_{\rm{univ}}. 
$$
Here $m$ is the dimension of $Y$. 
\vspace{0.2cm}\\
$(3)$ There exists a K\"ahler metric $g_{F_{\rm{univ}}}$ on ${F_{\rm{univ}}}$ 
such that the holomorphic sectional curvature of $g_{F_{\rm{univ}}}$ is semi-positive 
and the above isomorphism $X_{\rm{univ}} \cong \mathbb{C}^m \times F_{\rm{univ}}$ 
is isometric with respect to the K\"ahler metrics induced by $g$, $g_Y$,  and $g_{F_{\rm{univ}}}$. 
\end{theo}

\subsection*{Acknowledgements}
The author is indebted to Professor Junyan Cao 
for giving him an opportunity to consider holomorphic sectional curvature 
and Professor Xiaokui Yang for stimulating conversations on related topics. 
He would like to thank Professor Hajime Tsuji and Professor  Shigeharu Takayama 
for giving a helpful comment. 
He is supported by the Grant-in-Aid 
for Young Scientists (A) $\sharp$17H04821 from JSPS.

\section{Preliminaries}\label{sec2} 
In this section, we shortly summarize several results needed later by giving references. 
We can check basic properties of holomorphic sectional curvature in \cite{Mat18a} and \cite{Mat18b} with the same notation as in this paper. 

\subsection{Connections of vector bundles}\label{subsec2-1}
In this subsection, we recall the Gauss-Codazzi type formula for exact sequences of (holomorphic) vector bundles. 
We consider the exact sequence 
$$
0 \xrightarrow{\quad \quad} (S, g_S) \xrightarrow{\quad \quad}
(E, g) \xrightarrow{\quad  p \quad}(Q, g_Q)
\xrightarrow{\quad \quad} 0
$$ 
of vector bundles $S$, $E$, and $Q$ on a complex manifold $X$. 
We suppose that $E$ is equipped with a (smooth) hermitian metric $g$. 
Let $S^{\bot}$ be the orthogonal complement of $S$ in $E$ with respect to $g$, 
which is a $C^\infty$-bundle but not always a holomorphic vector bundle. 
The $C^\infty$-bundle $S^{\bot}$ is isomorphic to the quotient bundle $Q$ as $C^\infty$-bundles. 
More precisely,   there exists the $C^\infty$-bundle isomorphism $j: Q \cong S^{\bot}$ such that $p \circ j={\rm id}_Q$. 
In this paper, we often identify the orthogonal complement $S^{\bot}$  with the quotient bundle $Q$ under the isomorphism $j$. 
Then we have the orthogonal decomposition
\begin{align*}
E=S \oplus S^{\bot} = S \oplus j(Q) \cong S \oplus Q \text{ as $C^\infty$-bundles.}
\end{align*}
From this orthogonal decomposition, the hermitian metric $g_S$ (resp. $g_Q$) on $S$ (resp. $Q \cong S^{\bot}$) is induced. 
For the hermitian vector bundles $(E,g)$, $(S,g_S)$, $(Q, g_Q)$, 
we  respectively denote their Chern connection by $D$, $D_S$, $D_Q$  
and their Chern curvature by $\sqrt{-1}\Theta_{g}$, $\sqrt{-1}\Theta_{g_S}$, $\sqrt{-1}\Theta_{g_Q}$. 
Now we can define 
$$
A\in C^{\infty}(X, \Lambda^{1,0}\otimes \Hom(S, S^{\bot})) 
\text { and }
B\in C^{\infty}(X, \Lambda^{0,1}\otimes \Hom(S^{\bot}, S))
$$ 
to be 
$$
D(f)=D_S(f) + A (f) \text{ \quad and \quad } D(e)=B (e) + D_Q(e)  \
$$
for a (local) section $e $ of $S^{\bot}$ and a section $f $ of $S$. 
It is known that $A$ (equivalently $B$)  is  identically zero if and only if 
the above exact sequence admits the holomorphic orthogonal splitting, 
which means that $S^{\bot}$ is a holomorphic vector bundle and 
the above orthogonal decomposition gives the holomorphic decomposition of $E$. 
Further we also have  the following properties\,$:$
\begin{align}
\big \langle \sqrt{-1}\Theta_{g}(v, \bar v)(e), e \big \rangle_g + 
\big \langle B_{\bar v}  (e), B_{\bar v}  (e) \big \rangle_{g_S} &= 
\big \langle \sqrt{-1}\Theta_{g_Q}(v, \bar v)(e), e \big \rangle_{g_Q}, \label{eq-1} \\
\big \langle \sqrt{-1}\Theta_{g}(v, \bar v)(f), f \big \rangle_g - 
\big \langle A_v (f), A_v (f) \big \rangle_{g_Q} &= 
\big \langle \sqrt{-1}\Theta_{g_S}(v, \bar v)(f), f \big \rangle_{g_S} \notag 
%\notag\\
%\big \langle A_v (f), e \big \rangle_{g_Q}+\big \langle f, B_{\bar v} (e) \big \rangle_{g_S} &=0 \notag 
\end{align}
for a tangent vector $v \in T_X$, a vector $e \in S^{\bot}$, and a vector $f \in S$.

\subsection{Holomorphic foliations}\label{subsec2-2}

The theory of (holomorphic) foliation plays an important role in the proof of the main theorem. 
In this subsection, we recall some results of foliations needed later. 
Let $T_X$ be the (holomorphic) tangent bundle on a complex manifold $X$ and 
let $V$ be a (holomorphic) subbundle of $T_X$. 
The subbundle $V \subset T_X$ is called an {\textit{integrable foliation}} 
if the space of (local) sections of $V$ is closed under the Lie bracket 
$$
[\bullet , \bullet]: T_X \times T_X \to T_X. 
$$ 
Frobenius's theorem asserts the integrability of $V$ is equivalent to 
the condition that an arbitrary point in $X$ has an open neighborhood $U$ 
and a smooth morphism $\phi: U \to \mathbb{C}^{m}$ 
such that  $V$ coincides with the relative tangent bundle of $\phi$ (that is, $V=T_{U/\mathbb{C}^{m}}:= \Ker d\phi_*$), 
where $m:=\dim X - {\rm{rank}}\, V$. 
For the subbundle $V\subset T_X$, 
the Lie bracket determines the $\mathcal{O}_X$-module morphism $\Lambda^2 V \to T_X/V$, 
and thus we obtain the holomorphic section 
$$
F_V:=[\bullet, \bullet] \in H^{0}(X, {\mathit{Hom}}(\Lambda^2 V, T_X/V )) 
$$
of the sheaf ${\mathit{Hom}}(\Lambda^2 V, T_X/V )$. 
The subbundle $V$ is integrable if and only if the above section $F_V$ is identically zero. 
If the subbundle $V$ is integrable on a (non-empty) Zariski open set of $X$, 
then $V$ is integrable on the whole space $X$. 
Indeed,  the section $F_V$ is vanishing on the Zariski open set in this case, 
and thus it follows that $F_V$ should be identically zero on $X$ 
since the sheaf ${\mathit{Hom}}(\Lambda^2 V, T_X/V )$ is torsion free as $\mathcal{O}_X$-modules.

In the case of $V$ being integrable, we can obtain the quotient map $X \to X/\simm$
by using the equivalence relation defined as follows\,$:$
a point $p$ is equivalent to a point $q$ if and only if $p$ and $q$ can 
be connected by fibers of local morphisms $U \to \mathbb{C}^{m}$. 
A fiber of  the quotient map $X \to X/\simm$ is called a {\textit{leaf} of 
the integrable foliation $V \subset T_X$. 
We remark that leaves are not necessarily compact and 
the quotient space  $X/\simm$ does not necessarily admit a complex structure. 
However, under some additional assumptions, 
it can be shown that  $X/\simm$ is a complex space such that 
the quotient map $X \to X/\simm$ is a holomorphic map. 

We will use Lemma \ref{fol-lem} in the proof of Theorem \ref{key-thm} and 
Lemma \ref{Ehr-lem} (so-called Ehresmann's theorem) in the proof of Theorem \ref{mainn-thm}. 

\begin{lemm}[{\cite[Corollary 2.11]{Hor07}}]\label{fol-lem}
Let $X$ be a compact K\"ahler manifold and $W \subset T_X$ be an integrable foliation of $X$. 
Assume that at least one leaf of the foliation $W$ is compact and rationally connected.

Then there exists a smooth morphism $X \to Y$ such that $W=T_{X/Y}$. 
\end{lemm}

\begin{lemm}[{\cite[Section V]{CL} cf. {\cite[Theorem 3.17 ]{Hor07}}}]\label{Ehr-lem}
Let $\phi: X \to Y$ be a smooth morphism between compact complex manifolds $X$ and $Y$. 
Assume that there exists an integrable foliation $V \subset T_X$ such that $T_X= V \oplus T_{X/Y}$. 

Then $\phi: X \to Y$ is a locally trivial morphism. 
More precisely, there exists a representation $\rho : \pi_1(Y) \to \Aut(F)$ such that 
$X$ is isomorphic to $Y_{\rm{univ}} \times F/\pi_1(Y)$. 
Further the morphism 
$$
\mu: Y_{\rm{univ}} \times F_{\rm{univ}} \to Y_{\rm{univ}} \times F/\pi_1(Y) \cong X
$$ 
is the universal cover of $X$ and we have  
$$
\mu^{*}V = \pr_1^*(T_{Y_{\rm{univ}}}) \text{ and } 
\mu^{*}T_{X/Y} = \pr_2^*(T_{F_{\rm{univ}}}). 
$$ 
Here $Y_{\rm univ}$ and $F_{\rm univ}$ respectively denote the universal covers of $Y$ and $F$, 
and also $\pr_{i}$ denotes the projection of the product $Y_{\rm{univ}} \times F_{\rm{univ}}$ 
to the $i$-th component. 
\end{lemm}

In order to apply Lemma \ref{fol-lem}, we will construct foliations by using the following lemmas and Theorem \ref{key-thm}.

\begin{lemm}[{\cite[Ex. 5.8]{Har77}}]\label{Har-lem}
Let $\mathcal{H}$ be a coherent sheaf on a complex manifold $X$. 
Then the function 
$$X \ni p \longmapsto
\dim \mathcal{H}_{p} \otimes_{\mathcal{O}_{X,p}}\mathcal{O}_{X,p}/\mathfrak{m}_p \in \mathbb{Z}
$$
is upper semi-continuous with respect to $p \in X$. 
Further, when the above function is constant, 
the sheaf $\mathcal{H}$ is a locally free sheaf. 
\end{lemm}

\begin{lemm}[{\cite[Proposition 1.6, Corollary 1.4, Corollary 1.7]{Har80}}]\label{ref-lem}
%\ \\
%$(1)$ Let $X_o$ be a Zariski open set in a normal variety $X$ and 
%$\mathcal{H}$ be a reflexive coherent sheaf on $X_o$. 
%Then the push-forward $i_{*} \mathcal{H}$ by the natural inclusion $i: X_o \subset X$ 
%is a reflexive coherent on $X$. 
%\\
Let $\mathcal{F}$ and $\mathcal{G}$ are reflexive coherent sheaves on  a normal variety $X$. 
If there exists a Zariski open set $X_o \subset X$ such that 
$\codim (X \setminus X_o) \geq 2$ and $\mathcal{F} \cong \mathcal{G}$ on $X_o$, 
then we have $\mathcal{F} \cong \mathcal{G}$ on the ambient space $X$. 
\end{lemm}

\section{Proof of the main result}\label{sec3}

The purpose of this section is to give a proof of Theorem \ref{m-thm}. 
The strategy of the proof is as follows\,$:$
We will show that a suitably chosen MRC fibration $\phi: X \dashrightarrow Y$ of $X$ satisfies the desired properties in Theorem \ref{m-thm}. 
MRC fibrations of projective varieties are not uniquely determined 
since there are ambiguities in the choices of the birational models of its image $Y$ 
(see \cite{KoMM92} and \cite{Cam92} for MRC fibrations). 
Therefore one of the main difficulties in the proof is to show that a MRC fibration of $X$ can be chosen to be a morphism without indeterminacy locus.
For this purpose, for a given MRC fibration $\phi: X \dashrightarrow Y$, 
we first prove that the tangent bundle $T_X$ can be decomposed into $T_X \cong T_{X/Y} \oplus \phi^{*}T_Y$ 
on some Zariski open $X_o$ satisfying $\codim (X \setminus X_o) \geq 2$ (see Theorem \ref{key-thm}).  
From this holomorphic splitting, 
we can choose a MRC fibration $\phi': X \to Y'$ without indeterminacy locus 
(which may be different from the original one $\phi: X \dashrightarrow Y$),  
by applying the result in \cite{Hor07} 
for foliation with rationally connected leaves (see Theorem \ref{main-thm}). 
Finally, we obtain various properties in Theorem \ref{m-thm}, 
by developing techniques in \cite{Mat18b} and by using the theory of foliations  again (see Theorem \ref{mainn-thm}). 

We first prove the following splitting theorem on a Zariski open set by generalizing the argument of \cite[Theorem 1.7]{Mat18b}.

\begin{theo}\label{key-thm}
Let $X$ be a smooth projective variety 
admitting a K\"ahler metric  with semi-positive holomorphic sectional curvature, 
and let $\phi: X \dashrightarrow Y$ be a $($non-trivial$)$ MRC fibration to a smooth projective variety $Y$. 
We take a Zariski open set $X_o$ in $X$ such that the restriction of $\phi$  
$$
\phi_o:=\phi|_{X_o}: X_o \to Y_o:=\phi(X_o) \subset Y 
$$ 
is a morphism.

Then the morphism $\phi$ is a smooth morphism on $X_o$. 
Moreover the standard exact sequence 
$$
0 \to T_{X/Y} \to T_X \to \phi^{*} T_Y \to 0 \text{ on $X_o$}
$$
admits the holomorphic orthogonal splitting, 
that is, the orthogonal complement $T_{X/Y}^{\bot}$ is a holomorphic vector bundle on $X_o$ 
and there exists an isomorphism $j: \phi^{*} T_Y \cong T_{X/Y}^{\bot}$ on $X_o$ such that 
$$
T_X =  T_{X/Y} \oplus j (\phi^{*} T_Y) \cong T_{X/Y} \oplus \phi^{*} T_Y \text{ on } X_o. 
$$
\end{theo}

\begin{proof}[Proof of Theorem \ref{key-thm}]
The proof will be given by carefully refining the argument of \cite[Theorem 1.7]{Mat18b}, 
which is based on the idea of \cite{HW15} and techniques in \cite{Mat18a}. 
A part of the proof is essentially the same as \cite{Mat18b}, 
and thus we will explain the precise proof as we refer to \cite{Mat18b}.

For a non-trivial MRC fibration $\phi: X \dashrightarrow Y$ of a smooth projective variety $X$, 
we denote by the notation $Z_o$ the indeterminacy locus of $\phi$. 
It is sufficient to prove the conclusions for the Zariski open set  $X_o$ defined by $X_{o}:=X \setminus Z_o$ 
since $X_o$ is the maximal Zariski open set in $X$ 
such that the restriction 
$$
\phi_o:=\phi|_{X_o}: X_o \to Y_o:=\phi(X_o) \subset Y 
$$ 
is a morphism. 
Now we take a resolution $\tau: \Gamma \to X$ of the indeterminacy locus $Z_o$ of $\phi$ 
and denote by the notation $\bar \phi: \Gamma \to Y$ the morphism satisfying 
the following commutative diagram\,$:$
\begin{equation*}
\xymatrix@C=40pt@R=30pt{
 & \Gamma \ar[d]_\tau \ar[rd]^{\bar{\phi}\ \ }  & \\ 
\bar X \ar@{>}[r]^{\ \ \pi \ \ \ }& X \ar@{-->}[r]^{\phi \ \ \ }  &  Y.\\   
}
\end{equation*}
In the proof, we mainly consider the coherent sheaf $L$ on $X$ defined by 
$$
L:=\big( \tau_*  \mathcal{O}_X(\bar \phi^* K_Y) \big)^{\vee \vee}. 
$$
We remark that $L$ is actually an invertible sheaf since it is a reflexive sheaf of rank one (see \cite[Corollary 1.2, Proposition 1.9]{Har80}). 
We freely identify the locally free sheaves with the vector bundles and 
we denote by the notation $\mathcal{F}^{\vee}$ the dual of a coherent sheaf $\mathcal{F}$ in this paper. 
By the definition, 
the line bundle $L$ coincides with the usual pull-back $\phi^{*}K_Y$ on the Zariski open set $X_o$, 
and thus $L$  can be seen as the extension of the pull-back $\phi^{*}K_Y$ defined on $X_o$ 
to the ambient space $X$. 
From the sheaf morphism 
\begin{align*}
\bar \phi^* K_Y \xrightarrow{\quad d \bar \phi^{*} \quad} \Lambda^m \Omega_{\Gamma} = \Lambda^m T_{\Gamma} ^{\vee}\quad \text{ on } \Gamma 
\end{align*}
defined by the pull-back $d \bar \phi^{*}$ under $\bar \phi$, we can get the injective sheaf morphism 
$$
 L\xrightarrow{\quad f \quad} 
 \Lambda^m \Omega_{X}  = \Lambda^m T_{X} ^{\vee} \quad \text{ on } X
$$
by taking the push-forward $\tau_*$ under the modification $\tau$ and 
by using the formula 
$\tau_{*} \mathcal{O}_\Gamma(\Lambda^m \Omega_{\Gamma}) =\mathcal{O}_X(\Lambda^m \Omega_{X})$. 
Here $m$ is the dimension of $Y$.

Let $g$ be a K\"ahler metric of $X$ with the semi-positive holomorphic sectional curvature $H_g$. 
By using the (smooth) hermitian metric $\Lambda^m h$  on $ \Lambda^m \Omega_{X}$ 
induced by the dual metric $h:=g^{\vee}$ on the cotangent bundle $\Omega_{X}$, 
we define  the singular hermitian metric $H$ on $L$ to be  
$$
|e_L|_H:=|f(e_L)|_{\Lambda^m h} \text{ for a non-vanishing local section $e_L$ of $L$}, 
$$
where $f$ is the above sheaf morphism. 
The dual singular hermitian line bundle 
$$(L^\vee, G:=H^{\vee}:=H^{-1})
$$ 
plays a central role rather than $(L, H)$ in the proof (see \cite{Dem} for singular hermitian metrics). 
It is easy to see that the sheaf morphism $f$ is a bundle morphism 
(that is, the section $f(e_L)$ is also non-vanishing in our case) 
if and only if $G$ (equivalently $H$) is a smooth metric. 
In particular, since $f$ coincides with the usual pull-back $d \phi^{*}$ on $X_o$, 
we can conclude that $\phi$ is a smooth morphism on $X_o$ if $G$ can be shown to be smooth. 
This is our first goal.

We take a resolution $\pi: \bar{X} \to X$ of the degenerate ideal of $f$. 
Then, by the same argument as in \cite[Claim 1]{Mat18b}, 
it can be proven that there exist a smooth  $(1,1)$-form $\gamma$ on $\bar X$ and an effective divisor $E$ on $\bar X$ 
such that 
\begin{align}\label{eq-2}
2\pi c_1(\pi^{*}L^\vee) \ni  \pi^{*}\sqrt{-1}\Theta_G=\gamma + [E], 
\end{align}
where $[E]$ is the integration current defined by the effective divisor $E$ and $\sqrt{-1}\Theta_G$ is the curvature current of $(L^{\vee}, G)$ 
defined by $\sqrt{-1}\Theta_G:=\deldel \log |e_{L}|^{2}_{H}$. 
We will prove that the first Chern class of $L$ is  zero and  $\gamma $ is a semi-positive $(1,1)$-form, 
which implies that $E=0$ (in particular $\phi$ is smooth on $X_o$).

The image $Y$ of the non-trivial MRC fibration is not uniruled  by \cite{GHS03}, 
and thus the canonical bundle $K_Y$ of $Y$ is pseudo-effective by \cite{BDPP}.  
Hence, by the definition of $L$, it can be seen that the line bundle $L$ is also pseudo-effective. 
Then, for the K\"ahler form $\omega$ associated to the given metric $g$, 
the intersection number $c_1(\pi^* L) \cdot \{\pi^* \omega\}^{n-1}$ is non-negative. 
Therefore we obtain 
\begin{align}\label{eq-3}
0 \geq 2\pi \int_{\bar X} c_1(\pi^* L^\vee) \wedge \pi^* \omega^{n-1}
=\int_{\bar X} \gamma \wedge \pi^* \omega^{n-1} + \int_E \pi^* \omega^{n-1}
\end{align}
from (\ref{eq-2}) by taking the wedge product with $\pi^* \omega^{n-1}$ 
and the integration on $\bar X$. 
The second term of the right hand side is non-negative. 
Now we consider the first term. 
We can take a (non-empty) Zariski open set $Y_0$ such that 
the restriction $\phi :X_0:=\phi^{-1}(Y_{0}) \to Y_0$ 
is a proper smooth morphism since $\phi$ is an almost holomorphic map. 
By \cite[Claim 3.4]{Mat18b}, 
for an arbitrary point $p \in X_0$, 
there exists an orthonormal basis $\{e_{i}\}_{i=1}^n$ of the tangent space $T_{X, p}$ at $p$
satisfying the followings\,$:$
\begin{itemize}
\item[(1)] $\{e_i \}_{i=1}^{m}$ is an orthonormal basis of the orthogonal complement $(T_{X/Y, p})^{\bot}$. 
\item[(2)] $\Rur{g}{e_i}{e_j} \geq 0$ for any $1 \leq i, j \leq m$.
\item[(3)] $\sqrt{-1}\Theta_{G}(e_i, \bar e_i) \geq 0$   for any $i=1,2,\dots, m$. 
\end{itemize}
We remark that the curvature $\sqrt{-1}\Theta_{G}$ is a smooth $(1,1)$-form on a neighborhood of $p$  
since the metric $G$ is smooth on a neighborhood of a smooth point of $\phi$. 
Here the notation  
$$
R_g:=R_{(T_X,g)} \in 
C^{\infty}(X, \Lambda^{1,1}\otimes T_X^\vee \otimes \bar T_X^{\vee})
$$
denotes  the curvature tensor of $(X, g)$ defined by 
$$
R_g(x, \bar y,z, \bar u):= 
\big \langle \sqrt{-1}\Theta_{g}(x, \bar y)(z),  u\big \rangle_g
$$
for tangent vectors $x, y, z, u \in T_X$. 
Further, by the same argument as in \cite[Step 3]{Mat18b}, 
we can check the following equality\,$:$
\begin{align}\label{eq-4}
\frac{n}{2}\int_{\bar X} \gamma \wedge \pi^* \omega^{n-1}=
\int_{X_0} \sum_{i=1}^{m} \sqrt{-1}\Theta_{G}(e_i, \bar e_i) \, \omega^{n} +
\int_{X_0} \sum_{j=m+1}^{n} \sqrt{-1}\Theta_{G}(e_j, \bar e_j)\, \omega^{n}.   
\end{align}
Note that the integrand of the second term is the scalar curvature in the vertical direction 
(that is, the scalar curvature of a fiber $(F, g|_{F})$ of $\phi$), 
and thus it does not depend on the choice of the orthonormal basis $\{e_{i}\}_{i=1}^{n}$. 
The first term of the right hand side is non-negative by property $(3)$. 
The key point here is that the second term can be shown to be non-negative by Stokes's theorem and Fubini's theorem. 
We will omit the detailed proof  to avoid repeating (see \cite[Step 3]{Mat18b} for the precise proof). 
This implies that the right hand side is non-negative, 
and thus we can conclude that all the terms appearing in (\ref{eq-3}) and (\ref{eq-4}) should be equal to zero.

It follows that $\Theta_{G}(e_i, \bar e_i)=0$ for any $i=1,2, \dots, m$ 
since the first term in (\ref{eq-4}) is zero and its integrand is non-negative by property $(3)$. 
By applying the formula (\ref{eq-1}) to the exact sequence 
$$
0 \xrightarrow{\quad  \quad}   \Ker(d\phi_*) \xrightarrow{\quad  \quad}  \Lambda^m T_X \xrightarrow{\quad d\phi_* \quad} L^\vee = \phi^{*} K_Y^{\vee} \xrightarrow{\quad  \quad}  0
$$
on a neighborhood of $p \in X_0$, 
we can obtain 
\begin{align*}
0 \leq \sum_{k=1}^{m}\Rur{g}{e_i}{e_k}  \leq   \sqrt{-1}\Theta_{G}(e_i, \bar e_i)=0
\end{align*}
for any $i=1,2,\dots, m$. 
Here we used property (2) to obtain the left inequality. 
Each term in the left hand side is non-negative by property (2), 
and thus we can see that
$$ \text{
$H_g([e_i])=\Rur{g}{e_i}{e_i}=0$ for any $i=1,2,\dots, m$. 
}
$$ 
In particular, the tangent vector $e_i$ is the minimizer of 
the semi-positive holomorphic sectional curvature $H_g$. 
Hence it follows that $\Rur{g}{v}{e_i}$ is non-negative for any tangent vector $v \in T_X$
from \cite[Lemma 4.1]{Yan18a} (see also \cite[Lemma 2.1]{Mat18b}, \cite{Bre}, and \cite{BKT13}). 
By applying the formula (\ref{eq-1}) to a tangent vector $v \in T_X$ again, 
we  can see that 
\begin{align*}
0 \leq \sum_{i=1}^{m}\Rur{g}{v}{e_i} \leq \sqrt{-1}\Theta_{G}(v, \bar v). 
\end{align*}
This means that the curvature $\sqrt{-1}\Theta_{G}$ is semi-positive on $X_0$. 
On the other hand, we have $\gamma=\pi^{*}\sqrt{-1}\Theta_{G} $ on a (non-empty) Zariski open set of $\bar X$. 
Hence it follows that $\gamma$ is a semi-positive $(1,1)$-form on $\bar X$ since $\gamma$ is smooth on $\bar X$.

The first Chern class $c_{1}(\pi^* L^\vee)$ is represented by 
the sum of the semi-positive form $\gamma$ and the positive current $[E]$ by (\ref{eq-2}) and the above argument. 
On the other hand, the first Chern class $c_{1}(\pi^* L^\vee)$ is numerically zero 
since we have 
$$ 
\int_X c_1(L^\vee) \cdot \{\omega \}^{n-1}=\int_{\bar X} c_1(\pi^* L^\vee) \cdot \{\pi^* \omega \}^{n-1} = 0
$$ 
and $L$ is pseudo-effective. 
This implies  that $\gamma=0$ and $E=0$ (namely, $\sqrt{-1}\Theta_{G}=0$). 
Therefore we can conclude that 
$G$ is a smooth metric.
In particular, the morphism $\phi $ is a smooth morphism on $X_o$. 

In the rest of the proof, we will show that  
$$
B\in C^{\infty}(X, \Lambda^{0,1}\otimes \Hom(T_{X/Y}^{\bot}, T_{X/Y}))
$$
defined for the below exact sequence (see subsection \ref{subsec2-1}) is identically zero on $X_o$ 
(which leads to the desired decomposition of the tangent bundle $T_{X}$). 
By the above argument, we have already proved that   
$\Rur{g}{v}{e_i}$ is non-negative for any $i=1,2,\dots, m$ and any tangent vector $v \in T_{X}$. 
For the metrics $g_{S}$ and $g_{Q}$ on $T_{X/Y}$ and $\phi^{*} T_Y$ 
induced by  the standard exact sequence 
$$
0  \xrightarrow{\quad \quad}(T_{X/Y}, g_S) \xrightarrow{\quad \quad} (T_X, g)  \xrightarrow{\quad \quad} (\phi^{*} T_Y, g_Q)  \xrightarrow{\quad \quad} 0 \text{ on $X_o$}, 
$$  
we can  obtain 
\begin{align}\label{eq-5}
0 \leq \Rur{g}{v}{e_i} + 
\big \langle B_{\bar v} (e_i), B_{\bar v} (e_i) \big \rangle_{g_S}
= \big \langle \sqrt{-1}\Theta_{g_Q}(v, \bar v)(e_i), e_i \big \rangle_{g_Q}
\end{align}
by applying the formula (\ref{eq-1}) to the above exact sequence. 
On the other hand, the induced metric $\det g_Q$ on $\phi^{*}K_Y^{\vee}=\det \phi^{*} T_Y$ 
coincides with the metric $G$ on $X_o$ by the definition of $G$, 
and further the curvature $\sqrt{-1}\Theta_{G}$ of $(\phi^{*}K_Y^{\vee}, \det g_Q=G)$ is flat by the above argument. 
Therefore we have 
\begin{align*}
&\sum_{i=1}^{m}\big \langle \sqrt{-1}\Theta_{g_Q}(v, \bar v)(e_i), e_i \big \rangle_{g_Q}\\
=&
\big \langle \sqrt{-1}\Theta_{\det g_Q}(v, \bar v)(e_1 \wedge e_{2} \wedge \cdots \wedge e_{m}), e_1 \wedge e_{2} \wedge \cdots \wedge e_{m} \big \rangle_{\det g_Q}\\
=&0. 
\end{align*}
By combining it with the inequality (\ref{eq-5}), 
we can obtain that 
\begin{align*}
\Rur{g}{v}{e_i} =0 
\quad \text{ and } \quad  
\big \langle B_{\bar v} (e_i), B_{\bar v} (e_i) \big \rangle_{g_S} =0
\end{align*}
for  any $i=1,2,\dots, m$ and any tangent vector $v \in T_X$. 
Here we used the fact that 
$\langle B_{\bar v} (\bullet), B_{\bar v} (\bullet)  \rangle_{g_S}$  is a semi-positive  definite quadratic form on $T_{X/Y}^{\bot}$. 
We can see that $B$ is identically zero on $X_{o}$ 
since the trace $\sum_{i=1}^{m}\langle B_{\bar v} (e_i), B_{\bar v} (e_i)  \rangle_{g_S}$ is zero. 
Therefore we obtain the desired decomposition on $X_o$ (see subsection \ref{subsec2-1}). 
\end{proof}

As an application of Theorem \ref{key-thm} and Lemma \ref{fol-lem}, 
we have\,$:$

\begin{theo}[=Theorem \ref{main-thm}]\label{main-thm2}
Let $X$ be a smooth projective variety admitting a K\"ahler metric  with semi-positive holomorphic sectional curvature.

Then we can choose a MRC fibration $X \to Y$ of $X$ to be 
a morphism $($without indeterminacy locus$)$  to a smooth projective variety $Y$. 
\end{theo}
\begin{proof}[Proof of Theorem \ref{main-thm2}.]
We first take an arbitrary MRC fibration $\phi: X \dashrightarrow Y$ to a smooth projective variety $Y$. 
We will find a morphism $\phi': X \to Y'$ (which may be different from the original one $\phi: X \dashrightarrow Y$)
such that $\phi': X \to Y'$ is one of MRC fibrations of $X$, 
by using Theorem \ref{key-thm}, Lemma \ref{fol-lem}, and  a similar way to \cite{CH17}. 
Let $X_o$ be the maximal Zariski open set in $X$ 
such that the restriction $\phi_o$ of $\phi$  
$$
\phi_o:=\phi|_{X_o}: X_o \to Y_o:=\phi(X_o) \subset Y 
$$ 
is a morphism and let $i: X_o \hookrightarrow X$ be the natural inclusion. 
Note that the codimension of the Zariski closed set $X \setminus X_o$ is larger than or equal to two. 
In summary, we have the following commutative diagram\,$:$
\begin{equation*}
\xymatrix@C=40pt@R=30pt{
 & X_o \ar[d]_i \ar[rd]^{\phi_o\ \ }  & \\ 
& X \ar@{-->}[r]^{\phi \ \ \ }  &  Y.\\   
}
\end{equation*}
By Theorem \ref{key-thm}, 
the orthogonal complement $T_{X_o/Y_o}^{\bot}$ is a holomorphic vector bundle and 
there exists the isomorphism $j: \phi^{*} T_{Y_{o}} \cong T_{X_o/Y_o}^{\bot}$  such that 
\begin{align}\label{eq-6}
T_{X_{o}} =  T_{X_o/Y_o} \oplus j (\phi_{o}^{*} T_{Y_{o}}). 
\end{align}
Now we define the coherent sheaves $V$ and $W$ on $X$ 
by 
$$
V:=\big( i_* \mathcal{O}_{X_o}((j(\phi_{o}^*T_{Y_o}))) \big)^{\vee \vee} \text{ and }
W:=\big( i_* \mathcal{O}_{X_o}((T_{X_o/Y_o})) \big)^{\vee \vee}. 
$$
Note that $V$ and $W$ are reflexive sheaves on $X$ since they are defined by the double dual. 
Then it can be seen that $T_X = V \oplus W $ holds on $ X_o$ by the holomorphic decomposition (\ref{eq-6}). 
By lemma \ref{ref-lem} and $\codim (X\setminus X_0) \geq 2$, 
this  decomposition can be extended to the ambient space $X$, 
that is, we obtain the decomposition 
$$
T_X \cong V \oplus W \text{ on } X. 
$$

Now we will show that the subsheaf $W \subset T_X$ is actually a subbundle of $T_X$ 
and it determines an integrable foliation. 
The dimensions 
$$
\dim (W_p \otimes_{\mathcal{O}_{X,p}}\mathcal{O}_{X,p}/\mathfrak{m}_p) 
\text{ \quad and \quad} 
\dim (V_p \otimes_{\mathcal{O}_{X,p}}\mathcal{O}_{X,p}/\mathfrak{m}_p)
$$
are upper semi-continuous with respect to $p \in X$ by Lemma \ref{Har-lem} 
and the sum is equal to $n=\dim X$. 
Hence they should be constant. 
Then it follows that $V$ and $W$ are actually locally free sheaves from Lemma \ref{Har-lem}. 
In order to check the integrality of $W \subset T_X$, 
we consider the section 
$$
F_{W} \in H^{0}(X, {\mathit{Hom}}(\Lambda^2 W, T_X/W ))
$$
obtained from the Lie bracket (see subsection \ref{subsec2-2}). 
The foliation $W \subset T_X$ is integrable on $X_o$ 
since $\phi_o: X_o \to Y_o$ is smooth and $W=T_{X_o/Y_o}$ holds on $X_o$. 
In particular, the section $F_{W}$ is vanishing on $X_o$. 
Then it follows that $F_{W} $ is identically zero on $X$
since the sheaf ${\mathit{Hom}}(\Lambda^2 W, T_X/W )$ is torsion free 
as $\mathcal{O}_X$-modules. 
Hence we can see that the foliation $W \subset T_X$ is integrable. 

A general fiber of $\phi$ is a leaf of the integrable foliation $W \subset T_X$. 
Hence a general leaf is compact since $\phi: X \dashrightarrow Y$ is almost holomorphic 
and it is  rationally connected since $\phi: X \dashrightarrow Y$ is a MRC fibration. 
By applying Lemma \ref{fol-lem} to our situation, 
we can find a smooth morphism 
$\phi' : X \to Y'$  such that $W=T_{X/Y'}$. 

It remains to show that the morphism $\phi': X \to Y'$ is also a MRC fibration. 
A general fiber of $\phi: X \dashrightarrow Y$ is not only a leaf of $W$ 
but also a leaf of $T_{X/Y'}$. 
Therefore we can construct a birational map $\pi: Y \dashrightarrow Y'$ such that 
\begin{equation*}
\xymatrix@C=40pt@R=30pt{
 & X \ar@{-->}[d]_\phi \ar[rd]^{\phi' \ \ }  & \\ 
& Y \ar@{-->}[r]^{\pi \ \ \ }  &  Y'. 
}
\end{equation*}
This means that the morphism $\phi': X \to Y'$ is a MRC fibration. 
\end{proof}

By Lemma \ref{Ehr-lem} and the result of \cite{Mat18b}, 
we have\,$:$

\begin{theo}[=Theorem \ref{mainn-thm}]\label{mainn-thm2}
Let $(X, g)$ be a compact K\"ahler manifold with semi-positive holomorphic sectional curvature, and 
let $\phi: X \to Y$ be a morphism to a compact K\"ahler manifold $Y$ with pseudo-effective canonical bundle. 
\vspace{0.2cm}\\
$(1)$ The following statements hold\,$:$
\begin{itemize}
\item The morphism $\phi: X \to Y$ is smooth and locally trivial. 
\item The exact sequence of vector bundles 
$$
0 \xrightarrow{\quad \quad} T_{X/Y} 
\xrightarrow{\quad \quad} T_X \xrightarrow{\quad d\phi_* \quad} \phi^{*} T_Y 
\xrightarrow{\quad \quad}  0
$$
admits the holomorphic orthogonal splitting, 
that is, the orthogonal complement $T_{X/Y}^{\bot}$ is a holomorphic vector bundle and 
there exists an isomorphism $j: \phi^{*} T_Y \to T_{X/Y}^{\bot}$ 
such that it gives the holomorphic orthogonal decomposition
$$
T_X = T_{X/Y} \oplus j(\phi^{*} T_Y) \cong T_{X/Y} \oplus \phi^{*} T_Y.
$$ 
\item The image $Y$ admits a K\"ahler metric $g_{Y}$ 
such that $g_Q= \phi^{*} g_{Y}$ and the holomorphic sectional curvature of $g_Y$ is identically zero. 
In particular, the image $Y$ has a finite \'etale cover $T \to Y$ by a complex torus $T$. 
Here $g_Q$ is the hermitian metric on $\phi^{*} T_Y$ 
induced by the above exact sequence and the metric $g$. 
\end{itemize}
$(2)$ We obtain the isomorphism 
$$
X_{\rm{univ}} \cong \mathbb{C}^m \times F_{\rm{univ}}. 
$$
$(3)$ There exists a K\"ahler metric $g_{F_{\rm{univ}}}$ on ${F_{\rm{univ}}}$ 
such that the holomorphic sectional curvature of $g_{F_{\rm{univ}}}$ is semi-positive 
and the above isomorphism $X_{\rm{univ}} \cong \mathbb{C}^m \times F_{\rm{univ}}$ 
is isometric with respect to the K\"ahler metrics induced by $g$, $g_Y$,  and $g_{F_{\rm{univ}}}$. 
\end{theo}

\begin{proof}[Proof of Theorem \ref{mainn-thm2}]
Let $\phi: X \to Y$ be a morphism from a compact K\"ahler manifold $(X, g)$ 
with semi-positive holomorphic sectional curvature 
to a compact K\"ahler manifold $Y$ with pseudo-effective canonical bundle. 
Then, by \cite[Theorem 1.4]{Mat18b}, we can obtain the following properties\,$:$
\begin{itemize}
\item $\phi$ is a smooth morphism. 
\item There exists a K\"ahler metric  $g_Y$ on $T_Y$ such that 
$g_Q=\phi^* g_Y$ and the holomorphic sectional curvature of $(Y, g_Y)$ is identically zero. 
\end{itemize}
Here $g_Q$ is the metric on $\phi^{*} T_Y$ induced  by $g$ and the exact sequence 
\begin{align*}
0 \xrightarrow{\quad \quad } T_{X/Y}\xrightarrow{\quad \quad}  
T_X \xrightarrow{\quad d\phi_* \quad}
\phi^{*} T_Y
\xrightarrow{\quad \quad } 0. 
\end{align*}
The second property implies that  
the metric $g_{Y}$ is flat and the manifold $Y$ has a finite \'etale cover $T \to Y$ by a complex torus $T$ 
 (see \cite[Proposition 2.2]{HLW16}, \cite{Igu54}, and \cite{Ber66}). 
Moreover, by Theorem \ref{key-thm}, the above exact sequence admits the holomorphic orthogonal splitting, 
that is, there exists an isomorphism $j: \phi^{*} T_Y \cong  T_{X/Y}^{\bot} \subset T_X$ as holomorphic vector bundles  
such that it gives the orthogonal decomposition 
$$
T_X =T_{X/Y}  \oplus T_{X/Y}^{\bot} = T_{X/Y}  \oplus j (\phi^{*} T_Y). 
$$

By the above argument, we can see that 
it is enough for the statement (1) in Theorem \ref{mainn-thm2} 
to check that $\phi$ is locally trivial. 
For this purpose, we will show that the holomorphic subbundle 
$j (\phi^{*} T_Y) = T_{X/Y}^\bot \subset T_X$ determines an integrable foliation.  
Since $g$ is a K\"ahler metric,  
the Chern connection $D$ of $(T_X, g)$  corresponds to 
the Levi-Civita connection of the induced Riemannian metric. 
Hence, from the torsion freeness, we obtain that 
\begin{align}\label{eq-7}
[u_1, u_2]=D_{u_1}u_2 - D_{u_2}u_1
\end{align}
for any (local) vector fields $u_1, u_2$ on $X$. 
On the other hand, 
for a vector field $u$ in the horizontal direction (that is, it is a section of $j (\phi^{*} T_Y)=T_{{X/Y}}^\bot$), 
we have 
$$
D(u)=B (u) + D_Q(u)=D_Q(u). 
$$
Here  we used the fact that $B$ is identically zero, 
which can be obtained from the holomorphic orthogonal splitting of the above exact sequence. 
In particular, we can see that 
$D_{u_1}u_2$ and $ D_{u_2}u_1$ are also vector fields in the horizontal direction 
for any vector fields $u_{1}$ and $u_{2}$ in the horizontal direction. 
By the formula (\ref{eq-7}), 
we can conclude that the Lie bracket $[u_1, u_2]$ is in the horizontal direction. 
This means that  the horizontal tangent bundle $j (\phi^{*} T_Y)$ is integrable.

Therefore, by the Ehresmann theorem (see Lemma \ref{Ehr-lem}), we can see that $\phi$ is locally trivial. 
This verifies the statement (1). 
Further there exists a representation $\rho: \pi_{1}(Y) \to \Aut (F)$ such that 
$X \cong Y_{\rm{univ}} \times F/ \pi_{1}(Y)$. 
Then we have the following commutative diagram\,$:$
\begin{equation*}
\xymatrix@C=40pt@R=30pt{
F_{\rm{univ}} 
& X_{\rm univ}:=Y_{\rm{univ}} \times F_{\rm univ}
\ar[l]_{\pr_2\ \ \ \ \ \ \ \ } 
\ar[r]^{}  
\ar[dr]^{\pr_1} \ar@/^30pt/[rr]^{\mu }
& Y_{\rm{univ}} \times F
\ar[r]^{\ \ \ \ \ \ \ \ \ } \ar[d]_{} 
&Y_{\rm{univ}} \times F \big/ \pi_{1}(Y)\ar[d]^{\phi} \cong X\\ 
&
& Y_{\rm{univ}} \cong \mathbb{C}^{m}
\ar[r]^{ \pi }
& Y.\\   
}
\end{equation*}
Moreover the Ehresmann theorem asserts that the \'etale cover 
$$
\mu: Y_{\rm{univ}} \times F_{\rm univ} \to Y_{\rm{univ}} \times F \to  Y_{\rm{univ}} \times F/ \pi_{1}(Y) \cong X
$$
is the universal cover of $X$ 
 and we have 
\begin{align}\label{eq-8}
\pr_1^* T_{Y_{\rm univ}} =\mu^{*} j(\phi^{*}T_Y) \text{ \quad and \quad }
\pr_2^* T_{F_{\rm univ}} =\mu^{*} (T_{X/Y}). 
\end{align}
This verifies the statement (2).

It remains to show the  statement (3) in Theorem \ref{mainn-thm2}. 
Now we consider the induced metric $\mu^* g$ on $\mu^{*}T_X=T_{X_{\rm univ}}$ and 
the holomorphic decomposition 
$$
T_{X_{\rm univ}}=\pr_1^* T_{Y_{\rm univ}}\oplus \pr_2^* T_{F_{\rm univ}}=\mu^{*} j(\phi^{*}T_Y) \oplus \mu^{*} (T_{X/Y}). 
$$ 
The relative tangent bundle $T_{X/Y}$ is orthogonal to $j(\phi^{*}T_Y)=T_{X/Y}^{\bot}$ 
with respect to the metric $g$, 
and thus the above decomposition is also an orthogonal decomposition with respect to $\mu^* g$. 
Therefore it is sufficient to prove that $g_1$ and $g_2$ 
are respectively obtained from the pull-back of some metric of $T_{Y_{\rm univ}}$ and $ T_{F_{\rm univ}}$. 
Here $g_1$ (resp. $g_2$) is the metric on $\pr_1^* T_{Y_{\rm univ}}$ (resp. $\pr_2^* T_{F_{\rm univ}}$) induced by 
the metric $\mu^{*}g$ and the above decomposition.

We can easily prove that $g_1$ coincides with the pull-back of $\pi^{*} g_Y$ by  $\pr_1$. 
Indeed, we have $g_1=\mu^{*}g_Q$ by $\pr_1^* T_{Y_{\rm univ}} =\mu^{*} j(\phi^{*}T_Y) $ and 
$g_Q=\phi^{*} g_Y$ by the property of $g_Q$. 
Hence, by the above commutative diagram, we can easily check that 
$$
g_1=\mu^{*}(\phi^{*} g_Y)=\pr_1^*(\pi^* g_Y). 
$$

In the rest of the proof, we will show that  the metric $g_2$ on $\pr_2^* T_{F_{\rm univ}}$ can be obtained 
from the pull-back of some metric on $F_{\rm univ}$, in other words, 
the metric on the fiber $\pr_1^{-1}(y)$($\cong F_{\rm univ}$) defined by
the restriction of $\mu^{*}g$ and $\pr_2^* T_{F_{\rm univ}} |_{\pr_1^{-1}(y)} \cong  T_{F_{\rm univ}}$ 
is independent of $y \in Y_{\rm univ}$. 
For a local holomorphic vector field $w$ on $F_{\rm univ}$, 
we consider the section $\widetilde w:=\pr_2 ^* w$ of $\pr_2^* T_{F_{\rm univ}}$. 
Note that  the section $\widetilde w$ can be seen as a vector field in the vertical direction 
by the inclusion $\pr_2^* T_{F_{\rm univ}} \subset T_{X_{\rm univ}}$. 
If  the norm $|\widetilde w|_{\mu^* g}=|\widetilde w|_{g_2}$ is constant 
on a fiber $\pr_2^{-1}(p)$($\cong Y_{\rm{univ}}$) of the second projection $\pr_2$, 
we can easily see that the metric $g_{F_{\rm univ}}$ on $F_{\rm univ}$ defined by 
$|w|_{g_{F_{\rm univ}}}=|\widetilde w|_{\mu^* g}$ satisfies the desired property (that is, $g_2=\pr_2^* g_{F_{\rm univ}}$). 
In order to check that $|\widetilde w|_{\mu^* g}$ is constant on a fiber $\pr_2^{-1}(p)$, 
we will show that the differential of $|\widetilde w|_{\mu^* g}$ by the vector field $\widetilde v$ is identically zero 
for any (local) vector $v$ on $Y_{\rm univ}$. 
Here $\widetilde v$ is the vector field on $X_{\rm univ}$ defined by 
the section $\widetilde v:=\phi^* v$ of $\pr_1^* T_{Y_{\rm univ}}$ 
and the inclusion $\pr_1^* T_{Y_{\rm univ}} \subset T_{X_{\rm univ}}$. 
By the formula $\partial |\widetilde w|^2_{\mu^* g} =\langle  D \widetilde w,  \widetilde w\rangle_{\mu^* g}$, 
we obtain 
\begin{align*}
\langle  \partial |\widetilde w|^2_{\mu^* g},  \widetilde v\rangle_{\rm pair} = 
\langle  D_{\widetilde v} \widetilde w,  \widetilde w\rangle_{\mu^* g}. 
\end{align*}
On the other hand, we have 
$$
[\widetilde v, \widetilde w]=D_{\widetilde v}\widetilde w - D_{\widetilde w}\widetilde v
$$
by (\ref{eq-7}). 
The vector fields  $\widetilde v$ and $\widetilde w$ can be locally written as 
$$
\widetilde v=\sum_{i=1}^{m} a_i (z) \frac{\partial}{\partial z_i} \text{\quad and \quad} 
\widetilde w=\sum_{j=m+1}^{n} b_i(w) \frac{\partial}{\partial w_j}
$$
in terms of a local coordinate $z=(z_1,\dots, z_m)$ of $Y_{\rm univ}$ and 
a local coordinate $w=(w_{m+1},\dots, w_n)$ of $F_{\rm univ}$, 
since $\widetilde v$ (resp. $\widetilde w$) is constructed from the pull-back by $\pr_1$ (resp. $\pr_2$). 
From the above local expression and straightforward computations, 
we can easily check that $[\widetilde v, \widetilde w]=0$. 
Further $D\widetilde v$  and $D \widetilde w$ respectively preserve the horizontal and the  vertical  direction 
since the natural splitting $T_{X_{\rm univ}}=\pr_1^* T_{Y_{\rm univ}}\oplus \pr_2^* T_{F_{\rm univ}}$ 
is an orthogonal decomposition (see the argument in the second paragraph of this proof). 
Hence we can conclude that $D_{\widetilde v}\widetilde w=0$ and  $D_{\widetilde w}\widetilde v=0$. 
The differential of the norm $|\widetilde w|^2_{\mu^* g}$ in the horizontal direction is identically zero, 
and thus it is constant on a fiber of $\pr _2$. 
This finishes the proof. 
\end{proof}

We finally prove the main result of this paper, 
as a direct application of Theorem \ref{main-thm} and Theorem \ref{mainn-thm}.

\begin{theo}[=Theorem \ref{m-thm}]\label{m-thm2}
Let $X$ be a smooth projective variety 
equipped with a K\"ahler metric $g$ with semi-positive holomorphic sectional curvature. 
Then we have the following statements\,$:$
\vspace{0.2cm}\\
$(1)$ 
There exists a smooth morphism $\phi: X \to Y$ 
to a smooth projective variety $Y$ with the following properties\,$:$
\begin{itemize}
\item The morphism $\phi: X \to Y$ is a locally trivial morphism 
$($that is, all the fibers are isomorphic$)$. 
\item The image $Y$ is a smooth projective variety with a flat metric. 
In particular, there exists a finite \'etale cover $A \to Y$ by an abelian variety $A$. 
\item The fiber $F$ is a rationally connected manifold. 
In particular, the morphism $\phi: X \to Y$ is a MRC $($maximal rationally connected$)$ fibration of $X$. 
\end{itemize}
%Moreover the fiber product $\widetilde{X}:=X \times_Y A$ admits
%isomorphic to the product $A \times F$ of the abelian variety $A$ and 
%the rationally connected fiber $F$. 
Therefore the fiber product $X^*:=A \times_Y X$ admits 
the locally trivial Albanese map $X^* \to A$ to the abelian variety $A$ 
with the rationally connected fiber $F$. 
\vspace{0.2cm}\\
$(2)$ 
Moreover we obtain the isomorphism 
$$
X_{\rm{univ}} \cong \mathbb{C}^m \times F, 
$$
where $X_{\rm{univ}}$ is the universal cover of $X$ and 
$F$ is the rationally connected fiber of $\phi$. 
We have the following commutative diagram\,$:$
\begin{equation*}
\xymatrix@C=40pt@R=30pt{
 X_{\rm univ} \cong \mathbb{C}^{m} \times F\ar[r]^{} \ar[d]^{} & X^*:=A \times_Y X\ar[d] \ar[r]^{} & X\ar[d]^{\phi} \\ 
\mathbb{C}^{m} \ar[r]  & A \ar[r]^{}  &  Y.\\   
}
\vspace{0.1cm}
\end{equation*}
In particular, the fundamental group of $X$ is an extension of a finite group by $\mathbb{Z}^{2m}$. 
\vspace{0.2cm}\\
$(3)$ There exist a K\"ahler metric $g_F$ on the fiber $F$ and 
a K\"ahler metric $g_Y$ on $Y$ with the following properties\,$:$

\begin{itemize}
\item The holomorphic sectional curvature of $g_F$ is semi-positive.  
\item The K\"ahler metric $g_Y$ is flat. 
\item The above isomorphism $X_{\rm{univ}} \cong \mathbb{C}^m \times F$ is not only biholomorphic but also isometric 
with respect to the K\"ahler metrics $\mu^*g $, $\pi^*g_{Y}$, and $g_F$. 
\end{itemize}
Here $\pi$ and $\mu$ respectively denote the universal cover $\pi: \mathbb{C}^m \to Y$ of $Y$ 
and the universal cover $\mu: X_{\rm univ} \to X$ of $X$. 
\end{theo}

\begin{proof}[Proof of Theorem \ref{m-thm2}]
When the holomorphic sectional curvature is identically zero, 
the variety $X$ itself has a finite \'etale cover $A \to X$ by an abelian variety $A$ 
by \cite{Igu54} (see also \cite[Proposition 2.2]{HLW16} and \cite{Ber66}). 
In this case, there is nothing to prove since the identity map of $X$ satisfies the desired properties. 

When the holomorphic sectional curvature  is semi-positive but not identically zero, 
the canonical bundle $K_X$ is not pseudo-effective. 
Indeed, we have the equality 
$$
\int_{X} c_1(K_X) \wedge \omega^{n-1} = - \frac{1}{n\pi}\int_{X} S \, \omega^n, 
$$
where $\omega$ is the K\"ahler form associated to $g$ and $S$ is the scalar curvature of $g$. 
In our case, the scalar curvature $S$ is positive on a neighborhood of at least one point 
since $S$ can be described as the integral of the holomorphic sectional curvature. 
Therefore the right hand side is negative. 
This implies  that $X$ is uniruled by \cite{BDPP}. 
There is nothing to prove  in the case of $X$ being rationally connected. 

Therefore we may assume that $X$ admits a non-trivial MRC fibration. 
In this case, by Theorem \ref{main-thm}, 
we can take a MRC fibration $\phi: X \to Y$ of $X$ to be a morphism to a smooth projective variety $Y$. 
Further it can be seen that the canonical bundle $K_Y$ of $Y$ is pseudo-effective  by \cite{GHS03} and \cite{BDPP}.  
Hence the assumptions in Theorem \ref{mainn-thm} are satisfied. 
Then we obtain all the conclusions 
by applying Theorem \ref{mainn-thm} to the MRC fibration $\phi: X \to Y$ 
and by using  the fact that rationally connected manifolds are simply connected. 
\end{proof}

\begin{rem}\label{main-rem}
In the case of $X$ being a compact K\"ahler surface, 
the assumption of $X$ being projective can be removed 
by the classifications of surfaces (see the argument in \cite[Corollary 1.6]{Mat18a}). 
\end{rem}

%\newpage

%%%%%%%%%%%%%%%%%%%%%%

\end{document}